\documentclass[letterpaper, 11pt,  reqno]{amsart}

\usepackage{amsmath,amssymb,amscd,amsthm,amsxtra, esint}

\usepackage[implicit=true]{hyperref}

\allowdisplaybreaks[2]

\usepackage{color}
\definecolor{gr}{rgb}   {0.,   0.69,   0.23 }
\definecolor{bl}{rgb}   {0.,   0.5,   1. }
\definecolor{mg}{rgb}   {0.85,  0.,    0.85}
\definecolor{yl}{rgb}   {0.8,  0.7,   0.}
\definecolor{or}{rgb}  {0.7,0.2,0.2}

\newtheorem{theorem}{Theorem} [section]

\newtheorem{lemma}[theorem]{Lemma}
\newtheorem{proposition}[theorem]{Proposition}
\newtheorem{remark}[theorem]{Remark}

\newcommand{\I}{\hspace{0.5mm}\text{I}\hspace{0mm}}
\newcommand{\II}{\text{I \hspace{-2.8mm} I} }
\newcommand{\III}{\text{I \hspace{-2.9mm} I \hspace{-2.9mm} I}}

\newcommand{\IV}{\text{I \hspace{-2.9mm} V}}

\newcommand{\noi}{\noindent}
\newcommand{\Z}{\mathbb{Z}}
\newcommand{\R}{\mathbb{R}}
\newcommand{\C}{\mathbb{C}}
\newcommand{\T}{\mathbb{T}}

\let\Re=\undefined\DeclareMathOperator*{\Re}{Re}
\let\Im=\undefined\DeclareMathOperator*{\Im}{Im}

\newcommand{\RR}{\mathcal{R}}

\newcommand{\F}{\mathcal{F}}

\newcommand{\Nf}{\mathfrak{N}}
\newcommand{\Rf}{\mathfrak{R}}

\newcommand{\al}{\alpha}
\newcommand{\be}{\beta}

\newcommand{\Dl}{\Delta}
\newcommand{\eps}{\varepsilon}

\newcommand{\G}{\Gamma}
\newcommand{\ld}{\lambda}
\newcommand{\Ld}{\Lambda}
\newcommand{\s}{\sigma}

\newcommand{\ft}{\widehat}

\newcommand{\wt}{\widetilde}
\newcommand{\cj}{\overline}
\newcommand{\dx}{\partial_x}
\newcommand{\dt}{\partial_t}

\newcommand{\ta}{\theta}

\renewcommand{\l}{\ell}
\renewcommand{\o}{\omega}
\renewcommand{\O}{\Omega}

\newcommand{\les}{\lesssim}
\newcommand{\ges}{\gtrsim}

\newcommand{\jb}[1]
{\langle #1 \rangle}

\newcommand{\NN}{\mathcal{N}}

\newcommand{\GG}{\mathcal{G}}

\newcommand{\J}{\mathcal{J}}

\newtheorem*{ackno}{Acknowledgements}

\newcommand{\ub}{{\bf u}}
\newcommand{\uu}{\textsf{u}}
\newcommand{\w}{{\bf w}}

\newcommand{\Nb}{{\bf N}}
\newcommand{\Ns}{\textsf{N}}

\numberwithin{equation}{section}
\numberwithin{theorem}{section}

\usepackage[english, french]{babel}

\begin{document}
\baselineskip = 14pt

\selectlanguage{english}

\title[Quasi-invariant Gaussian measures for the cubic  third order NLS]
{Quasi-invariant Gaussian measures 
for the cubic 
nonlinear Schr\"odinger equation
with third order dispersion}

\author[T.~Oh, Y.~Tsutsumi,  and  N.~Tzvetkov]
{Tadahiro Oh, Yoshio Tsutsumi, and Nikolay Tzvetkov}

\address{
Tadahiro Oh, School of Mathematics\\
The University of Edinburgh\\
and The Maxwell Institute for the Mathematical Sciences\\
James Clerk Maxwell Building\\
The King's Buildings\\
Peter Guthrie Tait Road\\
Edinburgh\\ 
EH9 3FD\\
 United Kingdom}

\email{hiro.oh@ed.ac.uk}

\address{Yoshio Tsutsumi\\
Department of Mathematics\\ Kyoto University\\ Kyoto 606-8502\\ Japan}
\email{tsutsumi@math.kyoto-u.ac.jp}

\address{
Nikolay Tzvetkov\\
Universit\'e de Cergy-Pontoise\\
 2, av.~Adolphe Chauvin\\
  95302 Cergy-Pontoise Cedex \\
  France}

\email{nikolay.tzvetkov@u-cergy.fr}



\maketitle

\vspace{-8mm}

\begin{abstract}
In this paper, we consider the  cubic  nonlinear Schr\"odinger equation 
 with third order dispersion on the circle.
In the non-resonant case, 
we prove that the mean-zero Gaussian measures on Sobolev spaces 
$H^s(\T)$, $s > \frac 34$, 
are quasi-invariant under the flow.
In establishing the result, we apply gauge transformations to remove the resonant part of the dynamics
and use invariance of the Gaussian measures under these gauge transformations.

\end{abstract}

\begin{otherlanguage}{french}
\begin{abstract}


%

Dans cet article, nous consid\'erons l'\'equation   de Schr\"odinger non lin\'eaire  cubique avec dispersion d'ordre trois sur le cercle.
Dans le cas non r\'esonant, nous prouvons que les mesures gaussiennes de moyenne nulle sur les espaces de Sobolev
$ H ^ s (\T) $, $ s> \frac 34 $, sont quasi-invariantes par le flot.
En \'etablissant le r\'esultat, nous appliquons des transformations de gauge pour \'eliminer la partie r\'esonante de la dynamique
et nous utilisons l'invariance des mesures gaussiennes par rapport \`a ces transformations de gauge.

\end{abstract}
\end{otherlanguage}

\section{Introduction}
\label{SEC:intro}

\subsection{Cubic nonlinear Schr\"odinger equation with third order dispersion}

The main goal of this work is to extend the result of our previous paper \cite{OTz1}
to the more involved case of lower order dispersion.
Namely, we consider 
the following  
cubic  nonlinear Schr\"odinger equation with third order dispersion  (3NLS) on $\T$:
\begin{align}
\begin{cases}
i \dt u - i \dx^3 u - \be \dx^2 u  =  |u|^{2}u \\
u|_{t = 0} = u_0, 
\end{cases}
\quad  (x, t) \in \T\times \R, 
\label{3NLS1}
\end{align}

%
%

\noi
where  $u$ is a complex-valued function 
on $\T\times \R$ with $\T = \R/(2\pi \Z)$ and $\be \in \R$.
The equation \eqref{3NLS1} 
appears as a mathematical model for nonlinear pulse propagation phenomena in various fields of physics, 
in particular,  in nonlinear optics \cite{HK, Ag}.
The equation \eqref{3NLS1} without the third order dispersion
corresponds to the standard cubic nonlinear Schr\"odinger equation (NLS)
and it has been studied extensively from both theoretical and applied points of view.
In recent years, 
there has been an increasing interest in 
the cubic 3NLS \eqref{3NLS1} with the third order dispersion
in 
nonlinear optics \cite{Oi, LMKEHT, MS}.

While the equation \eqref{3NLS1}
conserves  the following Hamiltonian:
\begin{align*}
H(u) = - \frac 12 \Im \int_\T \dx^2 u \cj {\dx u}\, dx + \frac{\be}{2} \int_\T |\dx u |^2 dx - \frac{1}{4}\int_\T |u|^4 dx, 
\end{align*}

\noi
the leading order term is sign-indefinite 
and hence it does not play an important role in the well-posedness theory of \eqref{3NLS1}.
On the other hand, 
the conservation of the mass $M(u)$ defined by 
\begin{align*}
 M(u) = \int_\T|u|^2 dx, 
 \end{align*}

\noi
combined with local well-posedness in $L^2(\T)$, yields
the following global well-posedness  of~\eqref{3NLS1} in $L^2(\T)$.

\begin{proposition}\label{PROP:GWP}

The cubic  3NLS \eqref{3NLS1} is globally well-posed
in $H^s(\T)$
for $s \geq 0$.

\end{proposition}

The proof of local well-posedness
in $L^2(\T)$ follows from the Fourier restriction norm method
with the periodic Strichartz estimate.
See \cite{MT1}.
We point out that 
Proposition \ref{PROP:GWP} is sharp
since 
 \eqref{3NLS1} is ill-posed below $L^2(\T)$
 in the sense of non-existence of solutions~\cite{GO, MT2}.

In studying 3NLS \eqref{3NLS1} with the cubic nonlinearity, 
the following phase function $\phi(\bar n)$ plays an important role:
\begin{align}
\phi(\bar n)  & = \phi(n, n_1, n_2, n_3)
: = (n^3 -\be n^2) - (n_1^3 -\be n_1^2) 
+ (n_2^3 -\be n_2^2) - (n_3^3 -\be n_3^2) \notag\\
&= 3 (n - n_1) (n- n_3) \big(n_1 + n_3 - \tfrac23\be\big),
\label{phi1}
\end{align}

\noi
where the last equality holds under $n = n_1 - n_2 + n_3$.
Note that when $\frac{2\be}{3} \notin \Z$, the last factor never vanishes.
On the other hand, when $\frac{2\be}{3} \in \Z$, 
 the last factor is identically 0 for $n_3 = -n_1 + \frac{2\be}{3}$, $n_1 \in \Z$.
We  refer to the first case ($\frac{2\be}{3} \notin \Z$)
and the second case ($\frac{2\be}{3} \in \Z$)
as the non-resonant case and
the resonant case, respectively.
In the following, we focus on the non-resonant case.

\subsection{Transport property of the Gaussian measures on periodic functions}

Given $ s> \frac{1}{2}$, let $\mu_s$ be the  mean-zero Gaussian measure on $L^2(\T)$
with the covariance operator $2(\text{Id} - \Dl)^{-s}$, formally written as\footnote{Given a function $f$ on $\T$, 
we use both $\ft f_n$ and $\ft f(n)$ to denote the Fourier coefficient of $f$
at frequency~$n$.}
\begin{align}
 d \mu_s 
  = Z_s^{-1} e^{-\frac 12 \| u\|_{H^s}^2} du 
  =  \prod_{n \in \Z} Z_{s, n}^{-1}e^{-\frac 12 \jb{n}^{2s} |\ft u_n|^2}   d\ft u_n .
\label{gauss0}
 \end{align}

\noi
More concretely, 
we can define  $\mu_s$  as the induced probability measure
under the map\footnote{In the following, we drop the harmless factor of $2\pi$
when it does not play any important role.}
\begin{align}
 \o \in \O \mapsto u^\o(x) = u(x; \o) = \sum_{n \in \Z} \frac{g_n(\o)}{\jb{n}^s}e^{inx}, 
\label{gauss1}
 \end{align}

\noi
where  $\jb{\,\cdot\,} = (1+|\cdot|^2)^\frac{1}{2}$
and 
$\{ g_n \}_{n \in \Z}$ is a sequence of independent standard complex-valued 
Gaussian random variables (i.e.~$\text{Var}(g_n) = 2$) on a probability space
$(\O, \F, P)$.
It is easy to see that 
 $u^\o$ in \eqref{gauss1} lies  in $H^{\s}(\T)\setminus H^{s-\frac 12}(\T)$ for $\s < s -\frac 12$, almost surely. 
 Namely, 
 $\mu_s$ is a Gaussian probability measure on $H^{\s}(\T)$,
 $\s < s -\frac 12$.
  Moreover, for the same range of $\s$,  
 the triplet $(H^s, H^\s, \mu_s)$ forms an abstract Wiener space.  See  \cite{GROSS, Kuo2}.

Our main goal is to study the transport property
of the Gaussian measures $\mu_s$ on Sobolev spaces
under the dynamics of \eqref{3NLS1}.
We first recall the following definition of quasi-invariant measures.
Given a measure space $(X, \mu)$, 
we say that $\mu$ is {\it quasi-invariant} under a transformation $T:X \to X$
if 
the transported measure $T_*\mu = \mu\circ T^{-1}$
and $\mu$
are equivalent, i.e.~mutually absolutely continuous with respect to each other.

We now state our main result.

\begin{theorem}\label{THM:quasi}
Let $\frac 23 \beta \notin  \Z$.   Then, for $s > \frac 34$, the Gaussian measure $\mu_s$ is quasi-invariant under the flow of 
the  cubic 3NLS
 \eqref{3NLS1}. 
\end{theorem}

In probability theory, 
the transport property 
of Gaussian measures 
under linear and nonlinear transformations
has been studied extensively.
See, for example,  \cite{CM, Kuo, RA, Cru1, Cru2, Bog, AF}.
On the other hand,
in the field of Hamiltonian PDEs, 
Gaussian measures
naturally appear in the construction 
of invariant measures 
associated to conservation laws
such as Gibbs measures, 
  starting with the seminal work of Bourgain \cite{BO94, BO96}. 
See \cite{OTz1, BOP4} for the references therein.
In \cite{TzBBM}, the third author initiated the study of transport properties of Gaussian measures 
under the flow of a Hamiltonian PDE,
where two methods were presented
in establishing quasi-invariance of the Gaussian measures $\mu_s$
as stated in Theorem~\ref{THM:quasi}.
See also the subsequent work~\cite{OTz1, OTz3, OSTz} 
on the transport property of the Gaussian measures
under nonlinear Hamiltonian PDEs.

\smallskip

\noi
$\bullet$  {\bf Method 1:} 
The first method is to reduce an equation under consideration
so that one can apply a general 
criterion on quasi-invariance
of a Gaussian measure on an abstract Wiener space
under a nonlinear transformation
due to Ramer \cite{RA}.
Essentially speaking, 
this result  states that $\mu_s$
is quasi-invariant
if the nonlinear part is $(d+\eps)$-smoother than the linear part
for an evolution equation posed on $\T^d$.
Namely, the given nonlinear dynamics is basically a compact perturbation of the linear dynamics.

\smallskip

\noi
$\bullet$  {\bf Method 2:} 
This method was introduced in \cite{TzBBM} by the third author to go beyond Ramer's general argument
in studying concrete examples of evolution equations.
It is based on combining 
both PDE techniques 
and probabilistic techniques in an intricate manner.
In particular, 
the crucial step in this second method is to establish
an effective energy estimate (with smoothing)
for the (modified) $H^s$-functional.
See, for example,  Proposition 5.1 in~\cite{TzBBM} and  Proposition 6.1 in \cite{OTz1}.

\smallskip

We refer readers to \cite{OTz2}
for a brief introduction of the subject and an overview of these two methods.
We point out that, in applying either method, 
it is essential to exhibit nonlinear smoothing for given dynamics.
We also remark that 
the second method in general performs better than the first method.
See \cite{TzBBM, OTz1, OTz3}.
See also Remark \ref{REM:M2}.

In \cite{OTz1}, we studied the transport property of the Gaussian measure $\mu_s$
under 
the  following cubic fourth order  nonlinear Schr\"odinger equation (4NLS)  on $\T$:
\begin{align}
i \dt u -    \dx^4 u =  |u|^{2}u .
\label{4NLS}
\end{align}

\noi
Our main tool
to show nonlinear smoothing in this context was  normal form reductions
analogous to the approach employed in 
\cite{BIT, KO, GKO}.
In \cite{BIT}, Babin-Ilyin-Titi introduced
a normal form approach for constructing solutions to dispersive PDEs.
It turned out that this approach has various applications
such as establishing unconditional uniqueness~\cite{KO, GKO, CGKO}
and exhibiting nonlinear smoothing~\cite{ET}.
In applying  the first method,
we performed a normal form reduction 
to the (renormalized) equation
and
proved quasi-invariance of $\mu_s$ under \eqref{4NLS} for $s > 1$.
On the other hand, in applying the second method, 
we performed a normal form reduction 
to the equation satisfied by the (modified) $H^s$-energy functional
and proved quasi-invariance  for $s > \frac 34$
(which was later improved to the optimal range of regularity $s > \frac 12$
via an infinite iteration of normal form reductions
in \cite{OSTz}).

Now, let us turn our attention to  the cubic 3NLS \eqref{3NLS1}.
Let us first proceed as in~\cite{OTz1}
and transform the equation.
It is crucial that 
the Gaussian measure $\mu_s$ is quasi-invariant under 
the transformations we consider in the following.
(In fact, $\mu_s$ is invariant under these transformations.  See Lemma \ref{LEM:gauss}.)
Hence, it suffices to prove quasi-invariance of $\mu_s$ under the resulting dynamics.
Given $t \in \R$, we define a gauge transformation $\GG_t$ on $L^2(\T)$
by setting
\begin{align}
\GG_t [f ]: = e^{ 2 i t \fint |f|^2} f, 
\label{gauge1}
\end{align}

\noi
where $\fint_\T f(x) dx := \frac{1}{2\pi} \int_\T f(x)dx$.
Given a function $u \in C(\R; L^2(\T))$, 
we define $\GG$ by setting
\[\GG[u](t) : = \GG_t[u(t)].\]

\noi
Note that $\GG$ is invertible
and its inverse is given by $\GG^{-1}[u](t) = \GG_{-t}[u(t)]$.

Let  $u \in C(\R; L^2(\T))$ be a solution to \eqref{3NLS1}.
Define $\ub$ by 
\begin{align}
\ub (t) := \mathcal{G}[u](t)  = e^{ 2 i t \fint |u(t)|^2} u(t).
\label{gauge2}
\end{align}

\noi
Then, it follows from 
 the mass conservation 
that $\ub$ is a solution to the following renormalized 3NLS:
\begin{align}
i \dt \ub - i \dx^3 \ub - \be \dx^2 \ub   =  \bigg( | \ub |^{2}  - 2 \fint_\T |\ub |^2dx \bigg) \ub. 
\label{3NLS2}
\end{align}

\noi
Let $\Nb(\ub) = (|\ub|^{2} - 2 \fint_\T  |\ub|^2 dx\big)\ub $ be the renormalized nonlinearity in \eqref{3NLS2}.
Then, we have
\begin{align}
\Nb(\ub) 
& = \sum_{n\in \Z} e^{inx} \sum_{\substack{n = n_1 - n_2 + n_3\\n\ne n_1, n_3\\n_1 + n_3 \ne   \frac{2\be}{3}}}
\ft {\ub}_{n_1}\cj{\ft {\ub}_{n_2}}\ft {\ub}_{n_3}
 - \sum_{n\in \Z} e^{inx} |\ft {\ub}_n|^2\ft {\ub}_n
\notag\\
& \hphantom{X}
 + \sum_{n\in \Z} e^{inx}\sum_{\substack{n = n_1 - n_2 + n_3\\n\ne n_1, n_3\\n_1 + n_3 =    \frac{2\be}{3}}}
\ft {\ub}_{n_1}\cj{\ft {\ub}_{n_2}}\ft {\ub}_{n_3}
\notag\\
& =: \Nb_1 (\ub) + \Nb_2 (\ub) + \Nb_3 (\ub) .
\label{nonlin1}
\end{align}

\noi
In view of \eqref{phi1}, 
the first term corresponds to the non-resonant contribution, 
while the second and third terms correspond to the resonant contribution.
Moreover, under  the non-resonant assumption: $\frac{2\be}{3} \notin \Z$, 
we have $\Nb_3(\ub)\equiv 0$.
See Remark \ref{REM:res} for more on the renormalized equation \eqref{3NLS2}.

At this point, we can introduce  the interaction representation $v$ of $\ub$ as in \cite{OTz1} by 
\begin{align}
v(t) = S(-t)\ub(t), 
\label{gauge3}
\end{align}

\noi
where $S(t) = e^{t (\dx^3 - i \be \dx^2)}$
denotes the linear solution map for 3NLS \eqref{3NLS1}.
Under the non-resonant assumption ($\frac 23 \be \notin \Z$), this reduces 
\eqref{3NLS2} to the following equation for $\{\ft v_n\}_{n \in \Z}$:\footnote{The non-resonant part
$\NN_0(v)$ is non-autonomous.  For simplicity of notation, however, 
we drop the $t$-dependence.  A similar comment applies to multilinear terms
appearing in the following.} 
\begin{align}
\dt \ft v_n 
& = -i \sum_{\G(n)} e^{i t \phi(\bar n) } \ft  v_{n_1}\cj{\ft v_{n_2}}\ft v_{n_3}
+ i |\ft v_n|^2 \ft v_n \notag\\
& =: \ft{\NN_0(v)}(n) + \ft{\RR_0(v)}(n), 
\label{3NLS3}
\end{align}

\noi
where the phase function $\phi(\bar n)$ is as in \eqref{phi1} and the plane $\G(n)$ is given by
\begin{align}
\G(n) 
= \big\{(n_1, n_2, n_3) \in \Z^3:\, 
 n = n_1 - n_2 + n_3, \   n \ne  n_1, n_3, 
 \text{ and } n_1 + n_3 \ne   \tfrac{2\be}{3}\big\}.
\label{Gam1}
 \end{align}

\noi
In view of \eqref{phi1}, 
we refer to the first term $\NN_0(v)$ and the second term $\RR_0(v)$ 
on the right-hand side of \eqref{3NLS3}
as the non-resonant and resonant terms, respectively.
On the one hand we do not have any smoothing on $\RR_0(v)$ under a time integration.
On the other hand, 
  Lemma \ref{LEM:phase} below on the phase function $\phi(\bar n)$ shows that 
there is a smoothing on the non-resonant  term $\NN_0(v)$
under a time integration.
Hence by applying a normal form reduction as in ~\cite{OTz1}, 
we can  exhibit $(1+\eps)$-smoothing on the nonlinear part if $s > 1$.
See Lemma~\ref{LEM:Znonlin}.
Then, by invoking Ramer's result (Proposition~\ref{PROP:RA}), 
we conclude quasi-invariance of~$\mu_s$ 
under~\eqref{3NLS3} (and hence under~\eqref{3NLS1}; see Lemma \ref{LEM:gauss}) for $s > 1$.
We first point out that the regularity $s > 1$ is optimal with respect to this argument
(namely, applying Ramer's result in a straightforward manner)
due to the resonant part $\RR_0(v)$.
See  Remark~5.4 in~\cite{OTz1}.
Moreover, due to a weaker dispersion for 3NLS \eqref{3NLS1}
as compared to 4NLS \eqref{4NLS}, 
the second method applied to \eqref{3NLS3} based on an energy estimate
does not work
for any $s \in \R$.
Hence, a new idea is needed to go below $s = 1$.

To overcome this problem, we introduce 
another gauge transformation, which is the main new idea of this paper.
Given $t \in \R$, we define a  gauge transformation $\J_t$ on $L^2(\T)$
by setting
\begin{align}
 \J_t [f ]: = \sum_{n \in \Z} e^{  -i t |\ft f_n|^2 } \ft f_ne^{inx}.
\label{gauge4}
\end{align}

\noi
Define $\uu$ by 
\begin{align}
\uu(t) = \J_t[\ub(t)].
\label{gauge5}
\end{align}

\noi
First, by noting that $|\ft \uu_n(t)|^2 = |\ft \ub_n(t)|^2$,
it follows from \eqref{nonlin1} and \eqref{gauge5} that 
\begin{align} 
\dt \big(|\ft \uu_n |^2\big)
& =2 \Re (\dt \ft \ub_n \cj{\ft \ub_n})\notag\\
& = 
 2 \Im \bigg( \sum_{\G(n)} e^{i t \psi(\bar n)}
\ft {\uu}_{n_1}\cj{\ft {\uu}_{n_2}}\ft {\uu}_{n_3} \cj{\ft \uu_n}\bigg), 
\label{gauge5a}
\end{align}

\noi
where the (time-dependent) phase function $\psi(\bar n)$ is defined by 
\begin{align*}
\psi(\bar n)= \psi(n, n_1, n_2, n_3)(\uu) := - |\ft \uu_n|^2+  |\ft \uu_{n_1}|^2 - |\ft \uu_{n_2}|^2 + |\ft \uu_{n_3}|^2 .
\end{align*}

\noi
Then, 
we see that $\uu$ satisfies the following equation:
\begin{align}
i \dt \uu - i \dx^3 \uu - \be \dx^2 \uu   =  \Ns_1(\uu) + \Ns_2(\uu), 
\label{3NLS4}
\end{align}

\noi
where the nonlinearities $\Ns_1(\uu)$ and $\Ns_2(\uu)$ are given by 
\begin{align*}
\Ns_1(\uu) 
& = \sum_{n\in \Z} e^{inx} 
 \sum_{\G(n)} e^{i t \psi(\bar n)}
\ft {\uu}_{n_1}\cj{\ft {\uu}_{n_2}}\ft {\uu}_{n_3},\\
\Ns_2(\uu) 
& = 2t  \sum_{n\in \Z} e^{inx} \,  \ft \uu_n  \Im\bigg(
\sum_{\G(n)} e^{i t \psi(\bar n)}
\ft {\uu}_{n_1}\cj{\ft {\uu}_{n_2}}\ft {\uu}_{n_3}\cj{\ft \uu_n}\bigg).
\end{align*}

\noi
Finally, we consider the  interaction representation $w$ of $\uu$ given by 
\begin{align}
w(t) = S(-t)\uu(t). 
\label{gauge6}
\end{align}

\noi
Then, 
the equation 
\eqref{3NLS4} is reduced to the following equation for $\{\ft w_n\}_{n \in \Z}$: 
\begin{align}
\dt \ft w_n 
& = -i \sum_{\G(n)} e^{i t (\phi(\bar n) +\psi(\bar n)) } \ft  w_{n_1}\cj{\ft w_{n_2}}\ft w_{n_3}\notag\\
& \hphantom{Xl}
-2 i  t   \,  \ft w_n  \Im\bigg(
\sum_{\G(n)} e^{i t (\phi(\bar n) +\psi(\bar n)) }
\ft {w}_{n_1}\cj{\ft {w}_{n_2}}\ft {w}_{n_3}\cj{\ft w_n}\bigg)\notag\\
&  =: \ft{\NN_1(w)}(n)
 +  \ft{\NN_2(w)}(n),
\label{3NLS5}
\end{align}

\noi
where $\phi(\bar n)$ is  as in \eqref{phi1}
and  $\psi (\bar n)$ is now expressed in terms of $w$:
\begin{align}
\psi(\bar n)= \psi(n, n_1, n_2, n_3)(w) := - |\ft w_n|^2+  |\ft w_{n_1}|^2 - |\ft w_{n_2}|^2 + |\ft w_{n_3}|^2 .
\label{psi1}
\end{align}

\noi
By using the additional gauge transformation $\J_t$, 
 we removed the resonant part at the expense of introducing
the second term $\NN_2(w)$ in \eqref{3NLS5}.  
While this second term looks more complicated, 
it can be handled essentially in the same manner as  the non-resonant term $\NN_1(w)$
by noting that $\ft{\NN_2(w)}(n)$ is basically $\ft{\NN_1(w)}(n)$
with two extra (harmless) factors of $\ft w_n$. 
See Lemma \ref{LEM:Xnonlin}.
We also note that the phase function $\psi(\bar n)$  in \eqref{psi1} depends on 
the time variable $t$, 
which introduces extra terms 
 in  the normal form reduction step.
See~\eqref{Xnonlin1} and~\eqref{Ynonlin1} below.
The main point is, however, that there is no resonant contribution
in~\eqref{3NLS5}
and,  
as a result, we can show  $(1+\eps)$-smoothing on the nonlinear part 
 for $\frac 34 < s < 1$ (Lemma \ref{LEM:Xnonlin})
and  apply Ramer's result
to conclude quasi-invariance of the Gaussian measure~$\mu_s$.

Given $t, \tau \in \R$, let 
$\Phi(t):L^2\to L^2$ be the solution map for \eqref{3NLS1}
and 
$ \Psi_0(t, \tau)$ and $\Psi_1(t, \tau) :L^2\to L^2$ 
be the solution maps for \eqref{3NLS3} and \eqref{3NLS5}, respectively,
sending initial data at time $\tau$
to solutions at time $t$.\footnote{Note that \eqref{3NLS3}
and \eqref{3NLS5}
are non-autonomous.
As in \cite{OTz1},  this non-autonomy does not play an essential
role in the remaining part of the paper.
} 
When $\tau =0$, we  denote $\Psi_0(t, 0)$
and $\Psi_1(t, 0)$ by $\Psi_0(t)$ and $\Psi_1(t)$ for simplicity.
Then, 
it follows from \eqref{gauge2},    \eqref{gauge3},   \eqref{gauge5}, and   \eqref{gauge6} 
that 
\begin{align*}
\Phi(t) =  \GG_t^{-1} \circ S(t) \circ \Psi_0(t)
\qquad \text{and} \qquad 
\Phi(t) =  \GG_t^{-1} \circ \J_t^{-1}\circ S(t) \circ \Psi_1(t).
\end{align*}

\noi
As we pointed out above, 
the Gaussian measure $\mu_s$ is invariant under $S(t)$, $\GG_t$, and $\J_t$
(Lemma \ref{LEM:gauss})
and hence it suffices to prove quasi-invariance of $\mu_s$
under $\Psi_0(t)$ or $\Psi_1(t)$.
In applying Ramer's result, we view 
$\Psi_0(t)$ and $\Psi_1(t)$ as the identity plus a perturbation.
By writing   
\begin{align}
\Psi_0(t) = \text{Id} + K_0(t)
\qquad\text{and}\qquad
\Psi_1(t) = \text{Id} + K_1(t), 
\label{nonlin1c}
\end{align}

\noi
we show that $K_0(t)(u_0)$ and $K_1(t)(u_0)$ are $(1+\eps)$-smoother
than the random initial data $u_0$ distributed according to $\mu_s$
in appropriate ranges of regularities
(Lemmas \ref{LEM:Znonlin} and \ref{LEM:Xnonlin}).

We conclude this introduction with several remarks.
%
%
%
%
%
%
%
%
%
%

\begin{remark}\rm 
Dispersion is essential in establishing quasi-invariance of $\mu_s$ in Theorem \ref{THM:quasi}.
In \cite{OSTz}, the first and third authors with Sosoe
studied the transport property of $\mu_s$
under the following dispersionless model on $\T$:
\begin{align}
i \dt u   =  |u|^{2}u 
\label{NLSX}
\end{align}

\noi
In particular, they showed that $\mu_s$ is not quasi-invariant under 
the dynamics of \eqref{NLSX}.

\end{remark}

\begin{remark}\label{REM:M2} \rm 

(i) 
We point out that  the regularity restriction $s > \frac 34$
for the cubic 4NLS~\eqref{4NLS}
in~\cite{OTz1}
was optimal in a straightforward application of the second method
based on an energy estimate of the form:
\begin{align}
\frac{d}{dt} E(u) 
\leq C(\| u  \|_{L^2})
\| u \|_{H^{s - \frac 12 - \eps}}^{2-\theta}
\label{P0}
\end{align}

\noi
for some $\theta > 0$
and for any solution $u$ to \eqref{4NLS}.\footnote{We  point out that one can 
also close an argument by establishing an energy estimate \eqref{P0} with  $\theta = 0$.
See~\cite{OTz3}.}
Here,     $E(u) = \|u\|_{H^s}^2 + R(u)$ denotes a modified $H^s$-energy 
 with a suitable correction term $R(u)$ obtained via a normal form reduction
 applied to (the evolution equation satisfied by) $\|u\|_{H^s}^2$.
Note that, in \eqref{P0},  we are allowed to place  only (at most) two factors of $u$
in the $H^\s$-norm  with $\s = s - \frac 12 - \eps $
and need to place other factors in the conserved  (weaker)  $L^2$-norm.
In particular, the regularity restriction $s > \frac 34$ (i.e.~$\s > \frac 14$)
comes from the following estimate \cite[(6.14)]{OTz1}:
\begin{align}
\bigg\| \sum_{(m_1, m_2, m_3) \in \G(n_1)}
 \ft u_{m_1}\cj{\ft u_{m_2}}\ft u_{m_3}
\bigg\|_{\l^\infty_{n_1}} \les
 \| u \|_{H^\frac{1}{6}}^3
  \les \|u\|_{L^2}^{1+\theta}
 \| u \|_{H^\s}^{2-\theta}, 
\label{P1}
\end{align}

\noi
which holds for $\s > \frac 14$.
We point out that, 
in applying the first method based on Ramer's result, 
we can place all the factors in the $H^\s$-norm.
See Lemmas \ref{LEM:Znonlin}
and \ref{LEM:Xnonlin}

%
%

We stress that this regularity restriction $s > \frac 34$ can not be removed unless one applies an infinite iteration of normal
form reductions as in \cite{OSTz}, since, if we stop applying normal form reductions
within a finite number of steps, then we would need to apply  \eqref{P1}
to estimate the contribution from the trilinear terms added at the very last step.
The same restriction applies to the cubic 3NLS \eqref{3NLS1}.
Namely,  even if we apply the second method based on an energy estimate
to  the transformed equation \eqref{3NLS5},
we can expect, at best, the same regularity range $s > \frac 34$
as in Theorem \ref{THM:quasi},
not yielding any improvement over our proof of Theorem \ref{THM:quasi} based on the first method.

\smallskip

\noi
(ii)  If we apply the second gauge transformation \eqref{gauge5}
to the cubic 4NLS \eqref{4NLS} and apply the first method based on Ramer's argument, 
we can prove quasi-invariance of $\mu_s$ for $s >  \frac 23$.
While this is better than the regularity restriction $s > \frac 34$ in  \cite{OTz1},
this approach does not seem to yield an optimal result ($s > \frac 12$) as in \cite{OSTz}
in view of \eqref{XX1} below.
One way in this direction would be to apply an infinite iteration of normal form reductions
at the level of the equation as in \cite{GKO}.

\end{remark}

\begin{remark}\rm
After the completion of this paper, 
the second method based on the energy estimate
has been further developed in \cite{PTV, GOTW}.
In a recent preprint \cite{FT}, Forlano-Trenberth applied the  approaches developed in \cite{PTV, GOTW}
to study the cubic fractional NLS:
\begin{align}
i \dt u + ( -  \dx^2)^\al u  =  |u|^{2}u
\label{fNLS}
\end{align}

\noi
and showed that the Gaussian measure $\mu_s$ in \eqref{gauss0}
is quasi-invariant under the flow of \eqref{fNLS}
for 
\[s
> \begin{cases}
\max( \frac 23, \frac{11}{6} - \al), & \text{if }\al \geq 1
\rule[-3mm]{0pt}{0pt}
\\
\frac{10\al + 7}{12}, & \text{if }\frac 12 < \al <  1.
\end{cases}
\]

\noi
In particular, this shows quasi-invariance of $\mu_s$ under the standard NLS with the second order dispersion
for $s > \frac 56$.

\end{remark}

\begin{remark}\rm

In \cite{NTT}, the second author with Nakanishi and Takaoka
studied the low regularity well-posedness of the modified KdV equation on $\T$.
In particular, the following ``gauge'' transformation was used in \cite[Theorem 1.3]{NTT}
under an extra regularity assumption
$|n|^\frac{1}{2}\ft u_n (0)\in \l^\infty_n$:
\begin{align*}
 \wt \J_{\text{mKdV}} [u ](t): = \sum_{n \in \Z} e^{  -i n \int_0^t |\ft u_n(t')|^2 dt'} \ft u_n(t)e^{inx}
\end{align*}

\noi
to completely remove the resonant part of the dynamics.
Note that this extra regularity assumption was needed 
 to guarantee  the boundedness of the  transformation $\wt \J_{\text{mKdV}}$.
Instead of $\J_t$ in \eqref{gauge4}, 
one may be tempted to use an analogous ``gauge'' transformation $\wt \J$ defined by 
\begin{align}
 \wt \J [\ub ](t): = \sum_{n \in \Z} e^{  -i  \int_0^t |\ft \ub_n(t')|^2 dt'} \ft \ub_n(t)e^{inx}
\label{gauge7}
\end{align}

\noi
(for solutions $\ub$ to \eqref{3NLS2})
since it would produce a simpler equation  than \eqref{3NLS4} and \eqref{3NLS5}.
Note that, in our smooth setting, 
we do not need  an extra regularity assumption thanks to the global well-posedness in $L^2(\T)$
stated in Proposition \ref{PROP:GWP}.

We point out that one crucial ingredient in the proof of Theorem~\ref{THM:quasi}
is the invariance\footnote{In view of Lemma~\ref{LEM:gauss}~(iii), 
quasi-invariance would suffice.} of the Gaussian measure $\mu_s$ under the gauge transformation $\J_t$
defined in \eqref{gauge4}.
Note that the  transformation $\wt \J$ in \eqref{gauge7} depends the evolution on $[0, t]$, 
namely, it is not a well-defined gauge transformation on the phase space $L^2(\T)$.
Hence, studying the transport property 
of $\mu_s$  under $\wt \J$ (such as quasi-invariance) would already require 
the understanding of the transport property of $\mu_s$
under \eqref{3NLS2}
(in particular, in a time average manner which is highly non-trivial).\footnote{Even 
in a situation where we have an invariant measure, it is not at all trivial to know 
how random solutions at different times
are correlated.  In fact, it is an important open question to study
the space-time correlation of a random solution distributed by an invariant (Gibbs) measure
for a dispersive PDE.}
Therefore, while the transformation $\wt \J$
may be of use for the low regularity well-posedness theory, 
  it is not suitable for our analysis.

\end{remark}

\begin{remark}\label{REM:res}
\rm

In \cite{MT2},  the second author with Miyaji
 considered the Cauchy problem for 
 the renormalized 3NLS   \eqref{3NLS2}  in the non-resonant case: $\frac{2\be}{3} \notin \Z$.
By adapting the argument in \cite{TT},  
they proved local well-posedness
of  \eqref{3NLS2} in $H^s(\T)$, $ s > -\frac 16$.
	
It is of interest to study the transport property of the Gaussian measure
$\mu_s$ in the resonant case: $\frac{2\be}{3} \in \Z$.
In this case, we can write $\Nb_3(\ub)$ in \eqref{nonlin1} as 
\begin{align*}
\Nb_3(\ub) 
=  \sum_{n\in \Z} e^{inx}
\cj{\ft \ub\big(- n+\tfrac{2\be}{3}\big)}
\bigg\{ \sum_{\substack{n_1 \in \Z\\ n\ne n_1,   -n_1 + \frac{2\be}{3}}}
\ft \ub(n_1)\ft \ub\big(- n_1+\tfrac{2\be}{3}\big)\bigg\}.
\end{align*}

\noi
Since there is no dispersion to exploit on this term, 
we do not have a result analogous to Theorem \ref{THM:quasi}
in the resonant case.
We also point out that the well-posedness of the renormalized 3NLS \eqref{3NLS2}
in negative Sobolev spaces
is open in the resonant case.

\end{remark}

In the following, 
various constants depend on the parameter $\be \notin \frac{3}{2}\Z$ but we suppress its dependence
since $\beta$ is fixed.
In view of the time reversibility of the equation,
we only consider positive times.

\section{Proof of Theorem \ref{THM:quasi}}
\label{SEC:Pf}

In this section, we present the proof of Theorem \ref{THM:quasi}
under the non-resonant assumption $\frac 23 \be \notin \Z$.
Our basic approach is to apply Ramer's result
after exhibiting sufficient smoothing
on the nonlinear part.
As mentioned in Section \ref{SEC:intro}, 
we first transform the original equation~\eqref{3NLS1}
to~\eqref{3NLS3} or~\eqref{3NLS5}.
We then perform 
a normal form reduction
and establish  nonlinear smoothing
by exploiting the dispersion of the equation.

We first recall the precise statement 
of the main result in  \cite{RA}
for readers' convenience.
In the following, we use   $HS(H)$ to denote
the space of Hilbert-Schmidt operators on $H$
and $GL (H)$ to denote the space of invertible linear operators on $H$
with a bounded inverse.

\begin{proposition}[Ramer
\cite{RA}]\label{PROP:RA0}
Let $(H, E, \mu)$ be an abstract Wiener space,
where  
$\mu$ is the standard Gaussian measure on $E$.
Suppose that $T = \textup{Id} + K: U \to E$
be a continuous (nonlinear) transformation
from some open subset   $U\subset E$ into $E$
such that
\begin{itemize}
\item[\textup{(i)}] $T$ is a homeomorphism of $U$ onto an open subset of $E$.
\item[\textup{(ii)}]
We have $K(U) \subset H$
and $K:U \to H$ is continuous.

\item[\textup{(iii)}]
For each $x \in U$, 
the map $DK(x)$ is a Hilbert-Schmidt operator on $H$.
Moreover, $DK: x \in U \to DK(x) \in  HS(H)$ is continuous.

\item[\textup{(iv)}]
$\textup{Id}_H + DK(x) \in GL(H)$ for each $x \in U$.

\end{itemize}

\noi

\noi
Then, $\mu$ and $\mu\circ T $
are mutually absolutely continuous measures on $U$.
\end{proposition}

\subsection{Basic reduction}

We decompose the solution map $\Phi(t)$ to \eqref{3NLS1}
as  
\begin{align*}
\Phi(t) =  \GG_t^{-1} \circ S(t) \circ \Psi_0(t)
\end{align*}

\noi
for $ s> 1$
and 
\begin{align*}
\Phi(t) =  \GG_t^{-1} \circ \J_t^{-1}\circ S(t) \circ \Psi_1(t)
\end{align*}

\noi
for $ \frac 34 < s \leq 1$.
The following proposition shows that, 
in order to prove quasi-invariance of the Gaussian measure $\mu_s$ under $\Phi(t)$,
it suffices to establish its quasi-invariance 
under $\Psi_0(t)$ or $\Psi_1(t)$.

\begin{lemma}\label{LEM:gauss}
\textup{(i)}
Given  a  complex-valued mean-zero  Gaussian random variable $g$ with variance $\s$, i.e.~$g \in \NN_\C(0, \s )$, 
let $  Tg = e^{- i t |g|^2} g$ for some $t \in \R$.
Then, 	$Tg  \in \NN_\C(0, \s)$. 

\smallskip

\noi
\textup{(ii)}
Let $t \in \R$.  Then, the Gaussian measure $\mu_s$ defined in \eqref{gauss0} is invariant under 
the linear map 
$S(t) = e^{t (\dx^3 - i \be \dx^2)}$, the map $\GG_t$ in \eqref{gauge1}, and the map $\J_t$
in \eqref{gauge4}.

\smallskip

\noi
\textup{(iii)}
Let $(X, \mu)$ be a measure space.
Suppose that $T_1$ and $T_2$
are maps on $X$ into itself
such that $\mu$ is quasi-invariant under $T_j$ for each $j = 1, 2$.
Then, $\mu$ is quasi-invariant under  $T = T_1 \circ T_2$.

\end{lemma}

\begin{proof}
In view of 
Lemmas 4.1, 4.2, 4.4, and 4.5 in \cite{OTz1}, 
it remains to prove invariance of $\mu_s$ under $\J_t$.
Note that $\mu_s$ can be written as an infinite product of Gaussian measures:
\begin{align*}
\mu_s = \bigotimes_{n \in \Z} \rho_n,
\end{align*}

\noi
where $\rho_n$ is the probability distribution for $\ft u_n = \frac{g_n}{\jb{n}^s}$ defined in \eqref{gauss1}.
In particular, $\rho_n$ is a mean-zero Gaussian probability measure on $\C$ with variance $2 \jb{n}^{-2s}$.
Note that the action of $\J_t$ on $\ft u_n$ is given by $T$ in  Part (i),
which leaves the Gaussian measure $\rho_n$ invariant. 
Hence,
we conclude that $\mu_s$ is invariant under~$\J_t$.
\end{proof}

Fix $s > \frac 34$ and $\s < s - \frac{1}{2}$ sufficiently close to $s - \frac 12$.
First, recall that $\mu_s$ is a probability measure on $H^{\s}(\T)$.
Given $R>0$, let $B_R$ be the open ball of radius $R$ centered at the origin in 
$H^{\s}(\T)$.
We also recall
from \eqref{nonlin1c},  \eqref{3NLS3}, and \eqref{3NLS5} that 
\begin{align*}
\Psi_j(t) = \text{Id} + K_j(t), \quad j = 0, 1, 
\end{align*}

\noi
where $K_j(t)$ is given by 
\begin{align*}
K_0(t)(u_0) 
& = \int_0^t  \NN_0(v)(t')  dt'
+ \int_0^t   \RR_0(v)(t') dt'
 =: \Nf_0(v)(t) + \Rf_0(v)(t), \\ 
K_1(t)(u_0) 
& = \int_0^t  \NN_1(w)(t')  dt' + \int_0^t  \NN_2(w)(t')  dt'
 =: \Nf_1(w)(t)
 +  \Nf_2(w)(t),
\end{align*}

\noi
where $v$ and $w$ are the solutions to \eqref{3NLS3} and \eqref{3NLS5}
with initial data $u_0$.
Then, the proof of Theorem \ref{THM:quasi} is reduced to proving the following proposition, 
guaranteeing
the hypotheses of Ramer's result (Proposition \ref{PROP:RA0}).
See the proof
of Theorem~1.2 for $s > 1$ in \cite[Subsection~5.2]{OTz1}.

\begin{proposition}\label{PROP:RA}
Given $s > \frac34$, let $j = 0$ if $s > 1$ and $j = 1$ if $\frac 34 < s \leq 1$.
  Given $R > 0$, 
there exists $\tau = \tau(R) > 0$ such that,
for each $t \in  (0, \tau(R)]$, the following statements hold:

\begin{itemize}
\item[(i)] 

$\Psi_j(t)$ is a homeomorphism of $B_R$
onto an open subset of $H^{\s}(\T)$.

\smallskip

\item[(ii)]  
We have $K_j(t) (B_R) \subset H^s(\T)$ and $K_j(t): B_R \to H^s(\T)$ is continuous.

\smallskip

\item[(iii)] 
For each $u_0 \in B_R$, 
the map $DK_j(t)|_{u_0}$ is a Hilbert-Schmidt operator on $H^s(\T)$.
Moreover, 
$DK_j(t): u_0 \in B_R \mapsto DK_j(t)|_{u_0} \in HS(H^s(\T))$
is continuous.

\smallskip

\item[(iv)] $\textup{Id}_{H^s} + DK_j(t)|_{u_0} \in GL (H^s(\T))$ 
for each $u_0 \in B_R$.

\end{itemize}

\end{proposition}

Furthermore, arguing as in the proof of Proposition 5.3 in \cite{OTz1}, 
we see that Proposition~\ref{PROP:RA} follows
once we prove  the following nonlinear estimates (Lemmas \ref{LEM:Znonlin} and \ref{LEM:Xnonlin}),
exhibiting $(1+\eps)$-smoothing.
See also Remark \ref{REM:Ramer}.
The first  lemma
shows  nonlinear smoothing for the $v$-equation \eqref{3NLS3}.
In particular, when $\s > \frac 12$, 
this lemma exhibits nonlinear smoothing of order $1+\eps$, 
yielding  Proposition~\ref{PROP:RA} and hence Theorem \ref{THM:quasi} for $s > 1$.

\begin{lemma}\label{LEM:Znonlin}
Let $\s> \frac 12$.
Then, we have 
\begin{align}
\| \Nf_0(v)(t) \|_{H^{\s+2}} & \les
\|v(0)\|_{H^\s}^3 + \|v (t)\|_{H^\s}^3 + t 
\sup_{t' \in [0, t]} \|v (t')\|_{H^\s}^5, 
\label{Znonlin2}\\
\| \Rf_0(v)(t) \|_{H^{3\s}} & \les
  t  \sup_{t' \in [0, t]}  \|v (t')\|_{H^\s}^3.
 \label{Znonlin3}
\end{align}
	
\end{lemma}

The proof of Lemma \ref{LEM:Znonlin}
follows closely that of Lemma 5.1 in  \cite{OTz1}.
Namely, we apply a normal form reduction to \eqref{3NLS3}
and convert  the cubic non-resonant nonlinearity into 
a quintic nonlinearity (plus cubic boundary terms).
See \eqref{Znonlin1} below.
While we had a gain of two derivatives for 4NLS \eqref{4NLS} in \cite{OTz1}, 
it is not the case for our problem due to a weaker dispersion.
See Lemma \ref{LEM:phase}.
On the other hand, 
the resonant part $\Rf_0(v)$ is trivially estimated by $\l^2_n \subset \l^6_n$
as in \cite{OTz1}.
Note that the amount of smoothing for the resonant part $\Rf_0(v)$ 
is $2\s$, imposing the regularity restriction $\s > \frac 12$
in order to have $(1+ \eps)$-smoothing.

The next lemma 
shows  nonlinear smoothing in the context of  
 the $w$-equation \eqref{3NLS5},
 where the resonant part giving the regularity restriction is now removed.
As in the case of Lemma \ref{LEM:Znonlin}, 
we perform a normal form reduction.
However, more care is needed
due to the lower regularity under consideration.

\begin{lemma}\label{LEM:Xnonlin}
Let $\frac 14 < \s \leq \frac 12 $.
Then, we have 
\begin{align*}
\| \Nf_1(w)(t) \|_{H^{\s+1+}} 
& \les
 \sum_{j = 0}^2 t^j \sup_{t' \in [0, t]} \|w (t')\|_{H^\s}^{2j+3}, \\
\| \Nf_2(w)(t) \|_{H^{\s+1+}} 
& \les
 \sum_{j = 1}^3 t^j \sup_{t' \in [0, t]} \|w (t')\|_{H^\s}^{2j+3}.
\end{align*}
	
\end{lemma}

We present the proofs of Lemmas \ref{LEM:Znonlin} and \ref{LEM:Xnonlin} 
in the next subsections.
Before proceeding to the proofs of the nonlinear estimates, 
we state the following elementary lemma on the phase function $\phi(\bar n)$
under the non-resonant assumption  $\frac 23 \be \notin\Z$.

\begin{lemma}\label{LEM:phase}
 Let $\phi(\bar n)$ and $\G(n)$ be as in \eqref{phi1} and \eqref{Gam1}.
Then,  one of the following holds on $\G(n)$\textup{:}

\smallskip

\noi\begin{itemize}
\item[\textup{(i)}]
With $n_{\max} = \max(|n|, |n_1|, |n_2|, |n_3|)$, we have
\begin{align}
&|\phi(\bar n)| \ges n_{\max}^2 \ld,
\label{phi3} 
\end{align}

\noi
where $\ld = \min\big( |n - n_1|,  |n- n_3|,  |n_1 + n_3 - \tfrac23\be|\big)$.

\smallskip

\item[\textup{(ii)}]
$|n| \sim |n_1|\sim |n_2|\sim|n_3|$
and 
\begin{align}
&|\phi(\bar n)| \ges n_{\max} \Ld,
\label{phi4}
\end{align}

\noi
where $\Ld = \min\big( |n - n_1|  |n- n_3|,  
 |n- n_3||n_1 + n_3 - \tfrac23\be|, 
|n - n_1|   |n_1 + n_3 - \tfrac23\be|\big)
$.

\end{itemize}

\end{lemma}

The proof of Lemma \ref{LEM:phase} is immediate
from the factorization in \eqref{phi1}.
See also \cite[(8.21), (8.22)]{BO93no2} for a similar property of the phase function
for the modified KdV equation.

\subsection{Nonlinear estimate: Part 1}
In this subsection, we present the proof of Lemma~\ref{LEM:Znonlin}.
Fix $\s > \frac 12$.
By writing \eqref{3NLS3} in the integral form, we have
\begin{align*}
\ft v_n (t) 
 & = \ft v_n(0) 
 -i \int_0^t \sum_{\G(n)} 
 e^{it' \phi(\bar n) } \ft v_{n_1}\cj{\ft v_{n_2}}\ft v_{n_3}(t') dt' 
 +i \int_0^t | \ft v_{n}|^2\ft v_{n}(t') dt'\notag\\
&  = : \ft v_n(0) +\ft{\Nf_0(v)}(n, t)+\ft{\Rf_0(v)}(n, t).
\end{align*}

\noi
In view of Lemma \ref{LEM:phase}, 
we have a non-trivial oscillation 
caused by 
the phase function $\phi(\bar n)$
in 
the non-resonant part $\Nf_0(v)$. 
We exploit this fast oscillation by 
a normal form reduction, 
i.e.~integrating by parts:
\begin{align}
\ft {\Nf_0(v)}(n, t) 
& =  - \sum_{\G( n)}
\frac{ e^{it' \phi (\bar n) } }{\phi(\bar n)}
\ft v_{n_1}(t')\cj{\ft v_{n_2}(t')}\ft v_{n_3}(t')\bigg|_{t' = 0}^t 
+ \sum_{\G( n)}  \int_0^t 
\frac{ e^{it' \phi(\bar n) } }{\phi(\bar n)}
\dt(  \ft v_{n_1}\cj{\ft v_{n_2}}\ft v_{n_3})(t') dt' \notag\\
& =  -\sum_{\G(n)}
\frac{ e^{i t \phi( \bar n) } }{\phi(\bar n)}
\ft v_{n_1}(t)\cj{\ft v_{n_2}(t)}\ft v_{n_3}(t) 
+ \sum_{\G( n)}
\frac{ 1}{\phi(\bar n)}
\ft v_{n_1}(0)\cj{\ft v_{n_2}(0)}\ft v_{n_3}(0) \notag \\
& \hphantom{X}
+ 2    \int_0^t \sum_{\G(n)}
\frac{ e^{i t'\phi(\bar n) } }{\phi(\bar n)}
\big\{ \ft{\NN_0(v)}(n_1) + \ft{\RR_0(v)}(n_1)\big\}\cj{\ft v_{n_2}}\ft v_{n_3}(t') dt' \notag\\
& \hphantom{X}
+   \int_0^t \sum_{\G( n)}
\frac{ e^{i t' \phi(\bar n)  } }{\phi(\bar n)}
\ft  v_{n_1}\cj{\big\{\ft  {\NN_0(v)}(n_2) + \ft {\RR_0(v)}(n_2)\big\}}\ft v_{n_3}(t') dt'\notag\\
&  =:  \ft \I(n, t) - \ft  \I(n, 0)
+  \ft \II(n, t) +  \ft \III(n,  t). 
 \label{Znonlin1}
\end{align}

\noi
In view of Lemma \ref{LEM:phase}, 
the phase function $\phi(\bar n)$ appearing in the denominators
allows us to exhibit a smoothing for $\Nf_0(v)$.
In the computation above, we formally 
switched the order of the time integration
and the summation.
Moreover, we applied the product rule in time differentiation
at the second equality.
These steps can be justified,
provided $\s \geq \frac 16$.  See~\cite{OTz1} for details.

We now present the proof of 
Lemma \ref{LEM:Znonlin}.

\begin{proof}[Proof of Lemma \ref{LEM:Znonlin}]

We first estimate the non-resonant term $\Nf_0(v)$ in \eqref{Znonlin2}.
If  $\phi(\bar n)$ satisfies \eqref{phi3} in Lemma \ref{LEM:phase}, 
then we can proceed as in the proof of Lemma 5.1 in \cite{OTz1}
and establish \eqref{Znonlin2} 
since the proof of Lemma 5.1 in \cite{OTz1} only requires two gains of derivative from 
the phase function $\phi(\bar n)$ and the algebra property of $H^\s(\T)$, $\s > \frac 12$.
Moreover, the resonant term $\Rf_0(v)$ in \eqref{Znonlin3} 
can be estimated exactly as in \cite{OTz1} with $\l^2_n \subset \l^6_n$.
Hence, it remains to prove \eqref{Znonlin2}
under the assumption that 
 $\phi(\bar n)$ satisfies \eqref{phi4} in Lemma \ref{LEM:phase}.
In this case, we have less gain of derivative from $\phi(\bar n)$ in the denominator
and hence we need to proceed with more care.
Without loss of generality, assume
\begin{align*}
|\phi(\bar n)| \ges n_{\max}  |n - n_1|  |n- n_3|.
\end{align*}

\noi
Recall that we have $|n| \sim |n_1|\sim |n_2|\sim|n_3|$ in this case.

We  first consider the term $\I$.
It follows from Cauchy-Schwarz inequality that 
\begin{align}
\| \I (t)\|_{H^{3\s+1}} 
& 
\les \bigg\| \jb{n}^{3\s} \sum_{\G(n)}
\frac{1}{|n - n_1|  |n- n_3|}\prod_{j = 1}^3 |\ft v_{n_j}(t)| \bigg\|_{\l^2_n} \notag \\
& 
\les 
\sup_{n \in \Z} \bigg(\sum_{\G(n)}
\frac{1}{|n - n_1|^2  |n- n_3|^2}\bigg)^\frac{1}{2}
 \| v(t)\|_{H^\s}^3\notag \\
&  \les  \| v(t)\|_{H^\s}^3.
\label{Znonlin4}
\end{align}


Next, we consider $\II$.
The contribution $\II_\text{res}$ from $\RR_0(v)$ can be estimated as in \eqref{Znonlin4}.
With \eqref{3NLS3}, Cauchy-Schwarz inequality, and $\l^2_n \subset \l^6_n$, we have
\begin{align*}
\| \II_\text{res} (t)\|_{H^{5\s+1}} 
& 
\les t \sup_{t'\in [0, t]} 
\bigg\| \jb{n}^{5\s} \sum_{\G(n)}
\frac{1}{|n - n_1|  |n- n_3|}
|\ft v_{n_1}(t')|^3
\prod_{j = 2 }^3 |\ft v_{n_j}(t')| \bigg\|_{\l^2_n} \notag \\
& 
\les t\, 
\sup_{n \in \Z} \bigg(\sum_{\G(n)}
\frac{1}{|n - n_1|^2  |n- n_3|^2}\bigg)^\frac{1}{2}
\cdot  \sup_{t'\in [0, t]}  \| v(t')\|_{H^\s}^5\notag \\
&  \les  t \sup_{t'\in [0, t]}  \| v(t')\|_{H^\s}^5.
\end{align*}

\noi
We now estimate 
the contribution $\II_\text{nr}$ from $\NN_0(v)$ in $\II$.
Proceeding as in \eqref{Znonlin4} with the algebra property of $H^\s(\T)$, $\s > \frac 12$, we have
\begin{align*}
\| \II_\text{nr}(t) \|_{H^{3\s+1}} 
& 
\les t \sup_{t'\in [0, t]} 
\bigg\| \jb{n}^{3\s} \sum_{\substack{n = n_1 - n_2 + n_3\\ n \ne n_1, n_3} }
\frac{1}{|n - n_1|  |n- n_3|}
|\ft {\NN_0(v)}(n_1, t')|
\prod_{j = 2 }^3 |\ft v_{n_j}(t')| \bigg\|_{\l^2_n} \notag \\
& 
\les 
 t \sup_{t'\in [0, t]} \|  \NN_0(v)(t')\|_{H^\s} \| v(t')\|_{H^\s}^2\notag \\
&  \les  t \sup_{t'\in [0, t]}  \| v(t')\|_{H^\s}^5.
\end{align*}

\noi
The fourth term $\III$ in \eqref{Znonlin1} can be estimated in an analogous manner.
Finally, 
 by noting that $3 \s + 1 > \s + 2$ for $\s > \frac 12$, 
we conclude that the estimate \eqref{Znonlin2} holds for $\s > \frac 12$.
\end{proof}

\subsection{Nonlinear estimate: Part 2}
In this subsection, we present the proof of Lemma~\ref{LEM:Xnonlin}.
Fix $\s > \frac 14$.
By writing \eqref{3NLS5} in the integral form, we have
\begin{align}
\ft w_n (t) 
 & = \ft w_n(0) 
+ \int_0^t \ft{\NN_1(w)} (n, t') dt'
+  \int_0^t \ft{\NN_2(w)} (n, t') dt'\notag\\
&  = : \ft w_n(0) +\ft{\Nf_1(w)}(n, t)
+ \ft{\Nf_2(w)}(n, t).
\label{XNLS1}
\end{align}

\noi
As in the previous subsection, 
our basic strategy is to apply  a normal form reduction.
The main difference comes from the time dependent nature of 
the phase function $\psi(\bar n)$ in \eqref{psi1}.

For a short-hand notation, we set
\[ \ta(\bar n) = \phi(\bar n) + \psi (\bar n)\]

\noi
and
\begin{align}
\Xi(n, t) = 2\Im \bigg(\sum_{(m_1, m_2, m_3) \in \G(n)}
 e^{i t \ta(n, \bar m )}
 \ft w_{m_1}\cj{\ft w_{m_2}}\ft w_{m_3}
\cj{\ft w_{n}}\bigg),
\label{xi1}
\end{align}

\noi
where $(n, \bar m) = (n, m_1, m_2,m_3)$.
Then, 
from \eqref{3NLS5}, 
we have
\begin{align}
\ft {\Nf_1(w)}(n, t) 
& = 
-i \int_0^t \sum_{\G(n)} \dt \bigg(\frac{e^{i t' \phi(\bar n)}}{i \phi(\bar n)}\bigg) e^{i t' \psi(\bar n)} \ft  w_{n_1}\cj{\ft w_{n_2}}\ft w_{n_3}(t') dt' \notag\\
&  = -\sum_{\G(n)}
\frac{ e^{i t \ta( \bar n) } }{\phi(\bar n)}
\ft w_{n_1}(t)\cj{\ft w_{n_2}(t)}\ft w_{n_3}(t) 
+ \sum_{\G( n)}
\frac{ 1}{\phi(\bar n)}
\ft w_{n_1}(0)\cj{\ft w_{n_2}(0)}\ft w_{n_3}(0) \notag \\
& \hphantom{X} 
+i \int_0^t \sum_{\G(n)}
\frac{ e^{i t \ta( \bar n) } }{\phi(\bar n)}
\big\{ \psi(\bar n) + t' \dt \psi(\bar n)\big\} \ft  w_{n_1}\cj{\ft w_{n_2}}\ft w_{n_3}(t') dt' \notag\\
& \hphantom{X}
- 2i     \int_0^t \sum_{\substack{(n_1, n_2, n_3) \in \G(n)\\(m_1, m_2, m_3) \in \G(n_1)}}
\frac{ e^{i t' (\ta(\bar n) + \ta(n_1, \bar m ))} }{\phi(\bar n)}
( \ft w_{m_1}\cj{\ft w_{m_2}}\ft w_{m_3})
\cj{\ft w_{n_2}}\ft w_{n_3}(t') dt' \notag\\
& \hphantom{X}
+   i \int_0^t \sum_{\substack{(n_1, n_2, n_3) \in \G(n)\\(m_1, m_2, m_3) \in \G(n_2)}}
\frac{ e^{i t' (\ta(\bar n) - \ta(n_2, \bar m ))} }{\phi(\bar n)}
\ft  w_{n_1}(\cj{ \ft w_{m_1}}\ft w_{m_2}\cj{\ft w_{m_3}})
\ft w_{n_3}(t') dt'\notag\\
& \hphantom{X}
- 2i     \int_0^t t' 
\sum_{ \G(n)}
\frac{ e^{i t' \ta(\bar n) } }{\phi(\bar n)}
\Xi(n_1, t')
\ft w_{n_1} \cj{\ft w_{n_2}}\ft w_{n_3}
(t') dt' \notag\\
& \hphantom{X}
+  i  \int_0^t t' 
\sum_{ \G(n)}
\frac{ e^{i t' \ta(\bar n) } }{\phi(\bar n)}
\Xi(n_2, t')\ft  w_{n_1}\cj{\ft w_{n_2}}\ft w_{n_3}(t') dt'\notag\\
&  =:  \ft \I(n,  t) -  \ft  \I(n, 0)  + \ft \II(n,  t) +  \ft  \III_1(n,  t) \notag\\
& \hphantom{X}+ \ft  \III_2(n,  t) + \ft  \IV_1(n,  t) + \ft  \IV_2(n,  t).
\label{Xnonlin1}
\end{align}

\noi
As in \eqref{Znonlin1}, 
switching the order of  the time integration
and the summation in the computation above
can be justified for $\s \geq \frac 16$.
 See~\cite{OTz1}.
Similarly, we have 
\begin{align}
\ft {\Nf_2(w)}(n, t) 
& = 
-2 i \int_0^t t' \ft w_n \Im\bigg( \sum_{\G(n)} \dt \bigg(\frac{e^{i t' \phi(\bar n)}}{i \phi(\bar n)}\bigg) e^{i t' \psi(\bar n)} 
\ft  w_{n_1}\cj{\ft w_{n_2}}\ft w_{n_3}\cj{\ft w_n}\bigg)(t') dt' \notag\\
& = 
2 i  t \, \ft w_n \Re\bigg( \sum_{\G(n)} 
\frac{e^{i t \ta(\bar n)}}{  \phi(\bar n)}
\ft  w_{n_1}\cj{\ft w_{n_2}}\ft w_{n_3}\cj{\ft w_n}\bigg)(t)  \notag\\
& \hphantom{X} 
-2 i  \int_0^t \ft w_n \Re \bigg( \sum_{\G(n)} \frac{e^{i t' \ta(\bar n)}}{ \phi(\bar n)} e^{i t' \psi(\bar n)} 
\ft  w_{n_1}\cj{\ft w_{n_2}}\ft w_{n_3}\cj{\ft w_n}\bigg)(t') dt' \notag\\
& \hphantom{X} 
+ 2 i \int_0^t t' \ft w_n \Im\bigg( \sum_{\G(n)} 
\frac{e^{i t' \ta(\bar n)}}{ \phi(\bar n)}
\big\{ \psi(\bar n) + t' \dt \psi(\bar n)\big\}
\ft  w_{n_1}\cj{\ft w_{n_2}}\ft w_{n_3}\cj{\ft w_n}\bigg)(t') dt' \notag\\
& \hphantom{X} 
-2 i  \int_0^t t' (\dt \ft w_n) \Re \bigg( \sum_{\G(n)} \frac{e^{i t' \ta(\bar n)}}{ \phi(\bar n)} e^{i t' \psi(\bar n)} 
\ft  w_{n_1}\cj{\ft w_{n_2}}\ft w_{n_3}\cj{\ft w_n}\bigg)(t') dt' \notag\\
& \hphantom{X} 
-4 i  \int_0^t t' \ft w_n \Re \bigg( \sum_{\G(n)} \frac{e^{i t' \ta(\bar n)}}{ \phi(\bar n)} e^{i t' \psi(\bar n)} 
(\dt \ft  w_{n_1}) \cj{\ft w_{n_2}}\ft w_{n_3}\cj{\ft w_n}\bigg)(t') dt' \notag\\
& \hphantom{X} 
-2 i  \int_0^t t' \ft w_n \Re \bigg( \sum_{\G(n)} \frac{e^{i t' \ta(\bar n)}}{ \phi(\bar n)} e^{i t' \psi(\bar n)} 
\ft  w_{n_1} (\cj{\dt \ft w_{n_2}})\ft w_{n_3}\cj{\ft w_n}\bigg)(t') dt' \notag\\
& \hphantom{X} 
-2 i  \int_0^t t' \ft w_n \Re \bigg( \sum_{\G(n)} \frac{e^{i t' \ta(\bar n)}}{ \phi(\bar n)} e^{i t' \psi(\bar n)} 
\ft  w_{n_1} \cj{ \ft w_{n_2}}\ft w_{n_3}(\cj{\dt \ft w_n})\bigg)(t') dt' \notag\\
&  =:  \ft{\wt \I}(n,  t) + \ft {\wt  \II}(n,  t) + \ft{ \wt \III}(n,  t)  +  \ft {\wt \IV_0}(n,  t) \notag\\
& \hphantom{X} 
 +  \ft {\wt \IV_1} (n,  t)+  \ft{\wt \IV_2}(n,  t) + \ft{\wt  \IV_3}(n,  t), 
\label{Ynonlin1}
\end{align}

\noi
Modulo the extra phase factor $\psi(\bar n)$, 
the terms $\I, \III_1$, and $\III_2$ in~\eqref{Xnonlin1}
already  appear in~\eqref{Znonlin1}.
While the other terms in \eqref{Xnonlin1} and \eqref{Ynonlin1} are new, 
it turns out that they can be estimated
in a similar manner to $\I, \III_1$, and $\III_2$
with small modifications.

We now present the proof of 
Lemma~\ref{LEM:Xnonlin}.

\begin{proof}[Proof of Lemma \ref{LEM:Xnonlin}]
In the following, 
we first estimate  the terms $\I, \III_1$, and $\III_2$ in \eqref{Xnonlin1}.
We then show how the estimates for the other terms in \eqref{Xnonlin1}
and \eqref{Ynonlin1} follow from  those
for $\I, \III_1$, and $\III_2$.

\smallskip

\noi
{\bf Main argument:}
We first consider $\I$ in \eqref{Xnonlin1}.
If $\phi(\bar n)$ satisfies \eqref{phi4}, 
then we estimate $\I$ as in~\eqref{Znonlin4}.
Next, suppose that $\phi(\bar n)$ satisfies \eqref{phi3}.
Without loss of generality, assume that 
\begin{align}
|\phi(\bar n)| \ges n_{\max}^2 |n-n_1|.
\label{phi6}
\end{align}

\noi
Then, by Cauchy-Schwarz inequality
with $\jb{n}^{\s} \les \max_{j = 1, 2, 3}\jb{n_j}^{\s}$ for $\s \geq 0$, we have
\begin{align}
\| \I (t)\|_{H^{\s+1+}} 
& 
\les \bigg\| \jb{n}^{\s} \sum_{\substack{n = n_1 - n_2 + n_3\\ n \ne n_1, n_3} }
\frac{1}{|n - n_1| n_{\max}^{1-}}\prod_{j = 1}^3 |\ft w_{n_j}(t)| \bigg\|_{\l^2_n} \notag \\
& 
\les 
\sup_{n \in \Z} \bigg(\sum_{\substack{n = n_1 - n_2 + n_3\\ n \ne n_1, n_3} }
\frac{1}{|n - n_1|^2  n_{\max}^{2-}}\bigg)^\frac{1}{2}
 \| w(t)\|_{H^\s}^3\notag \\
&  \les  \| w(t)\|_{H^\s}^3.
\label{Ynonlin1a}
\end{align}


Next, we consider the fourth term $\III_1$ in \eqref{Xnonlin1}.
The fifth term $\III_2$ in \eqref{Xnonlin1} can be estimated in an analogous manner.
In the following, we only estimate the integrand of $\III_1$.
With  abuse of notation, we also denote the integrand as $\III_1$.

\smallskip

\noi
$\bullet$ {\bf Case (i):}
$\phi(\bar n)$ satisfies \eqref{phi3}.
\\
\indent
We first consider the case that \eqref{phi6} holds.
In this case, 
for $\frac 16 < \s \leq \frac 12$, 
we have
\begin{align}
\frac{\jb{n}^{\s+1+}}{|\phi(\bar n)|}\frac{1}{\jb{n_1}^{\s - \frac 16}\jb{n_2}^{\s}\jb{n_3}^{\s}}
\les \frac{1}{\jb{n}^{\frac 12+}  |n-n_1|^{1-}  \jb{n_{2}}^{\s}\jb{n_{3}}^{\frac{1}{2}+}}.
\label{XX1a}
\end{align}
	
\noi
By the triangle inequality: $\jb{n_1}^{\s - \frac 16}
\les \max_{j = 1, 2, 3}\jb{m_j}^{\s - \frac 16}$
and 
Sobolev's inequality, we have 
\begin{align}
\bigg\| \jb{n_1}^{\s - \frac{1}{6}}\sum_{(m_1, m_2, m_3) \in \G(n_1)}
 \ft w_{m_1}\cj{\ft w_{m_2}}\ft w_{m_3}
\bigg\|_{\l^\infty_{n_1}} \les
 \| w \|_{H^\frac{1}{6}}^2
  \| w \|_{H^\s}
  \leq 
 \| w \|_{H^\s}^3
\label{XX1}
\end{align}

\noi
for $\s \geq \frac 16$.
Then, 
it follows from Cauchy-Schwarz inequality with  \eqref{XX1a} and \eqref{XX1} that 
\begin{align}
\| \III_1 \|_{H^{\s+1+}}
& \les \| w \|_{H^\s}^3
\bigg\|
 \sum_{(n_1, n_2, n_3) \in \G(n)}
\frac{1}{\jb{n}^{\frac 12+}  |n-n_1|^{1-}  \jb{n_{2}}^{\s}\jb{n_{3}}^{\frac{1}{2}+}}
\prod_{j = 2}^3 \jb{n_j}^\s |\ft w_{n_j}|\bigg\|_{\l^2_n} \notag\\
& \les \| w \|_{H^{\s}}^5
\bigg(\sum_{n, n_3\in \Z}
 \sum_{\substack{n_2 \in \Z\\n_2 \ne n_3}}
\frac{1}{\jb{n}^{1+} |n_2 - n_3|^{2-}  \jb{n_{2}}^{2\s}\jb{n_{3}}^{1+}}
\bigg)^\frac{1}{2}\notag 
\intertext{By first summing in $n_2$, then in $n_3$ and $n$,}
& 
 \les \| w \|_{H^{\s}}^5
\label{XX2}
\end{align}

\noi
for $\frac 16 < \s \leq \frac 12$.
The upper bound $\s \leq \frac 12$ is by no means sharp but it suffices
for our purpose.

Next, suppose that
\begin{align*}
|\phi(\bar n)| \ges n_{\max}^2 |n_1 + n_3 - \tfrac 23\be|
= n_{\max}^2 |n + n_2 - \tfrac 23\be|.
\end{align*}

\noi
In this case, we have 
\begin{align*}
\frac{\jb{n}^{\s+1+}}{|\phi(\bar n)|}\frac{1}{\jb{n_2}^{\s}\jb{n_3}^{\s}}
\les \frac{1}{\jb{n}^{\frac 12+}  \jb{n + n_2 - \tfrac 23\be}^{1-} \jb{n_{2}}^{\s-}\jb{n_{3}}^{\frac{1}{2}+}} .
\end{align*}
	
\noi
Then, we can repeat a computation analogous to \eqref{XX2}
once we notice 
\begin{align*}
\bigg(\sum_{n, n_3\in \Z}
 \sum_{n_2 \in \Z}
\frac{1}{\jb{n}^{1+}   \jb{n + n_2 -\tfrac 23\be}^{2-} \jb{n_{2}}^{2\s-}\jb{n_{3}}^{1+}} 
\bigg)^\frac{1}{2}< \infty.
\end{align*}

\noi
Similarly, we can handle the case
\begin{align*}
|\phi(\bar n)| \ges n_{\max}^2 |n-n_3|
\end{align*}

\noi
by noting
\begin{align*}
\frac{\jb{n}^{\s+1+}}{|\phi(\bar n)|}\frac{1}{\jb{n_2}^{\s}\jb{n_3}^{\s}}
\les \frac{1}{\jb{n}^{\frac 12+}  \jb{n -n_3}^{1-} \jb{n_{2}}^{\frac{1}{2}+}\jb{n_{3}}^{\s-}} 
\end{align*}

\noi
and 
\begin{align*}
\bigg(\sum_{n, n_2\in \Z}
 \sum_{n_3 \in \Z}
\frac{1}{\jb{n}^{1+}   \jb{n -n_3}^{2-} \jb{n_{2}}^{1+}\jb{n_{3}}^{2\s-}} 
\bigg)^\frac{1}{2}< \infty.
\end{align*}

\smallskip

\noi
$\bullet$ {\bf Case (ii):}
$\phi(\bar n)$ satisfies \eqref{phi4}.
\\
\indent
In this case, we have $|n| \sim |n_1| \sim |n_2| \sim |n_3|$.
We first consider the case
\[|\phi(\bar n)| \ges n_{\max} |n-n_1| |n-n_3|.\]

\noi
By Cauchy-Schwarz inequality, we have
\begin{align}
\| \III_1 \|_{H^{\s+1+}}
& \les 
\bigg\|
 \sum_{(n_1, n_2, n_3) \in \G(n)}
\frac{1}{\jb{n_1}^{\s-}  |n-n_1| |n-n_3|}\notag\\
& \hphantom{XXXXXXXX}
\times  \sum_{(m_1, m_2, m_3) \in \G(n_1)}
\prod_{i = 1}^3|\ft w_{m_i}|
\prod_{j = 2}^3 \jb{n_j}^\s |\ft w_{n_j}|\bigg\|_{\l^2_n} \notag\\
& \les \| w \|_{H^{\s}}^2
\bigg\|
\frac{1}{\jb{n_1}^{\s-}  \jb{n-n_1} \jb{n-n_3}}
 \sum_{(m_1, m_2, m_3) \in \G(n_1)}
\prod_{i = 1}^3|\ft w_{m_i}|\bigg\|_{\l^2_{n, n_1, n_3}}\label{XX3}
\intertext{By summing in $n_3$ and then in $n$
and
applying H\"older inequality (in $n_1$) and Young's inequality
(with $\frac{1-2\s+}{2} + 2 = \frac{1}{q}+\frac{1}{q}+\frac{1}{q}$),}
& \les \| w \|_{H^{\s}}^2
\bigg\| \sum_{(m_1, m_2, m_3) \in \G(n_1)}
\prod_{i = 1}^3|\ft w_{m_i}|\bigg\|_{\l^{\frac{2}{1-2\s+}}_{n_1}}
\les \| w \|_{H^{\s}}^2
\| \ft w_n \|_{\l^q_n}^3  \notag
\intertext{By H\"older's inequality,} 
& 
\les \| w \|_{H^{\s}}^2
\big(\| \jb{n}^{-\frac{2-q}{2q}-}\|_{\l^\frac{2q}{2-q}_n} 
\| \jb{n}^{\frac{2-q}{2q}+} \ft w_n\|_{\l^2_n}\big)^3\notag\\
& \les \| w \|_{H^{\s}}^2
 \| w \|_{H^{\frac{1-\s+}{3}}}^3 
\leq \| w \|_{H^{\s}}^5, \notag
\end{align}

\noi
provided that 
$\frac{1-\s+}{3} \leq \s$.
This 
gives  the regularity restriction
$\s > \frac 14$
stated in the hypothesis.

When $|\phi(\bar n)| \ges n_{\max} |n-n_1| |n_1+ n_3 - \tfrac{2}{3}\be|$
or
$|\phi(\bar n)| \ges n_{\max} |n-n_3| |n_1+ n_3 - \tfrac{2}{3}\be|$, 
we can proceed as in the computation above
but we need to first sum in $n$ and then in $n_3$ the corresponding factors
at \eqref{XX3}.

\smallskip

\noi
{\bf Remaining terms:}
We now estimate the remaining terms.
The main idea is to reduce the estimates to the main argument
presented above (for $\I, \III_1$, and $\III_2$)
by noticing that 
the extra factors appearing in the remaining terms are all bounded in the $\l^\infty_n$-norm.
From \eqref{psi1}, we have
\begin{align}
\| \psi(\bar n)\|_{\l^\infty_{\bar n}}
\les \| \ft w_n\|_{\l^\infty_n}^2
\leq \|  w\|_{L^2}^2, 
\label{Ynonlin2}
\end{align}

\noi
where  $\l^\infty_{\bar n} = \l^\infty_{n, n_1, n_2, n_3}$.
On the one hand,  it follows from 
\eqref{gauge5a}, \eqref{psi1}, and \eqref{xi1} that 
\begin{align}
\dt \psi(\bar n) = - \Xi(n) + \Xi(n_1)- \Xi(n_2)+ \Xi(n_3).
\label{Ynonlin2a}
\end{align}

\noi
On the other hand, 
from \eqref{XX1} with $\s = \frac 16$, we have
\begin{align}
\| \Xi(n) \|_{\l^\infty_{ n}}
\les \| \ft w_n\|_{\l^\infty_n} \|w \|_{H^\frac{1}{6}}^3
\leq\|w \|_{H^\frac{1}{6}}^4 .
\label{Ynonlin2b}
\end{align}

\noi
Combining \eqref{Ynonlin2a} and \eqref{Ynonlin2b}, we obtain
\begin{align}
\| \dt \psi(\bar n)\|_{\l^\infty_{\bar n}}
\les \|w \|_{H^\frac{1}{6}}^4 .
\label{Ynonlin3}
\end{align}

\noi
Hence, from the estimate \eqref{Znonlin4} and \eqref{Ynonlin1a} on $\I$
with \eqref{Ynonlin2} and \eqref{Ynonlin3}, 
we can estimate $\II$ in~\eqref{Xnonlin1} by 
\begin{align*}
\| \II (t) \|_{H^{\s+1+}} \les t \sup_{t' \in [0, t]} \|w(t') \|_{H^\s}^5 + t^2\sup_{t' \in [0, t]} \|w(t') \|_{H^\s}^7
\end{align*}

\noi
for $\s \geq \frac 16$.
Noting that the terms $\IV_j$, $j = 1, 2$, in \eqref{Xnonlin1}
 have the same structure
as $\I$ with an extra factors of $\Xi(n_j)$, 
it follows from the estimates on $\I$ and \eqref{Ynonlin2b} that 
\begin{align*}
\|  \IV_j (t) \|_{H^{\s+1+}} & \les t^2 \sup_{t' \in [0, t]}\|w(t') \|_{H^\s}^7
\end{align*}

\noi
for $j = 1, 2$.
Similarly, by noting that $\wt\I$, $\wt\II$, and $\wt \III$ in \eqref{Ynonlin1} basically have the same structure
as $\I$ in \eqref{Xnonlin1}
with two extra factors of $\ft w_n$
(and $ \psi(\bar n) + t' \dt \psi(\bar n)$ for $\wt \III$), 
we have
\begin{align*}
\| \wt \I (t) \|_{H^{\s+1+}} & \les t \|w(t) \|_{H^\s}^5, 
\\
\| \wt \II (t) \|_{H^{\s+1+}} & \les t \sup_{t' \in [0, t]} \|w(t') \|_{H^\s}^5,  
\\
\| \wt \III (t) \|_{H^{\s+1+}} & \les t^ 2 \sup_{t' \in [0, t]}\|w(t') \|_{H^\s}^7 + t^3  \sup_{t' \in [0, t]}\|w(t') \|_{H^\s}^9.
\end{align*}

\noi
As for $\wt \IV_0$ and $\wt \IV_3$ in \eqref{Ynonlin1}, we first observe that 
\begin{align}
\| \dt \ft w_n\|_{\l^\infty_n} 
\les \| w\|_{H^\frac{1}{6}}^3
+ t \|\ft w_n\|_{\l^\infty_n}^2\| w\|_{H^\frac{1}{6}}^3
 \les \| w\|_{H^\frac{1}{6}}^3
+ t\| w\|_{H^\frac{1}{6}}^5, 
\label{Ynonlin4}
\end{align}

\noi
which follows from \eqref{3NLS5} and \eqref{XX1}.
Then, by noting that $\wt \IV_0$ and $\wt \IV_3$  are basically $\I$
with extra factors of $\dt \ft w_n$ and $\ft w_n$, we obtain 
\begin{align*}
\| \IV_j (t) \|_{H^{\s+1+}}  \les t^ 2 \sup_{t' \in [0, t]}\|w(t') \|_{H^\s}^7 + t^3  \sup_{t' \in [0, t]}\|w(t') \|_{H^\s}^9
\end{align*}

\noi
for $j = 0, 3$.

It remains to consider $\wt \IV_j$, $j = 1, 2$, in \eqref{Ynonlin1}.
Note that they are basically $\III_j$ in \eqref{Xnonlin1}, 
where we replaced
\begin{align*}
\sum_{(m_1, m_2, m_3) \in \G(n_j)} e^{it' \theta(n_j, \bar m)}
 \ft w_{m_1}\cj{\ft w_{m_2}}\ft w_{m_3}
\end{align*}

\noi
by $\dt \ft w_{n_j}$ and added two extra factors of $\ft w_n$.
By a small modification of  \eqref{Ynonlin4}, we have
\begin{align}
\| \jb{n_j}^{\s - \frac 16} \dt \ft w_{n_j}\|_{\l^\infty_{n_j}} 
 \les \| w\|_{H^\s}^3
+ t\| w\|_{H^\s}^5
\label{Ynonlin6}
\end{align}

\noi
for $\s \geq \frac 16$.
Then, by repeating the computation in Case (i) above with \eqref{Ynonlin6}, we obtain 
\begin{align}
\| \IV_j (t) \|_{H^{\s+1+}}  \les t^ 2 \sup_{t' \in [0, t]}\|w(t') \|_{H^\s}^7 + t^3  \sup_{t' \in [0, t]}\|w(t') \|_{H^\s}^9
\label{Ynonlin7}
\end{align}

\noi
for $j = 1, 2$ when $\phi(\bar n)$ satisfies \eqref{phi3}.
Lastly, noting from \eqref{3NLS5} 
that the contribution from $\ft{ \NN_2(w)}(n_j)$ in $\dt \ft w_{n_j}$
is basically $\ft{ \NN_1(w)}(n_j)$ with two extra factors of $\ft w_{n_j}$.
Hence, 
by repeating the computation in Case (ii), 
we also obtain \eqref{Ynonlin7}
 when $\phi(\bar n)$ satisfies \eqref{phi4}.

This completes the proof of Lemma \ref{LEM:Xnonlin}.
\end{proof}

\begin{remark}\label{REM:Ramer}\rm
As mentioned above, once we have Lemma \ref{LEM:Xnonlin}, 
we can prove Proposition~\ref{PROP:RA}
with $j = 1$ and $\frac 34 < s \leq 1$
by repeating the proof of Proposition 5.3 in \cite{OTz1}.
In this case, we need to interpret the nonlinear part 
$K_1(t)(u_0) 
 = \Nf_1(w)(t) +  \Nf_2(w)(t)$ of the dynamics~\eqref{XNLS1}
as those given by the right-hand sides
of \eqref{Xnonlin1} and \eqref{Ynonlin1}.
In particular, in computing the derivative $DK_j(t)|_{u_0}$
for $u_0 \in B_R \subset H^\s(\T)$,
we need to take derivatives of the complex exponentials such as $e^{i t \psi(\bar n)}$
since $\psi(\bar n)$ depends on $w$.
While this introduces extra terms, 
it does not cause any issue since
such derivatives can be easily bounded.
For example, let $F(t)  = e^{i t \psi(\bar n)}$.
Then, with \eqref{psi1}, we have 
\begin{align*}
DF(t)|_{u_0}(\w(0)) 
& =  it F(w)(t)  D\psi(\bar n) (t)|_{u_0}(\w(0))\\
& =  2 it F(w)(t) \Re(  - \ft w_n \cj{\ft \w_n}
+ \ft w_{n_1} \cj{\ft \w_{n_1}}
- \ft w_{n_2} \cj{\ft \w_{n_2}}+ \ft w_{n_3} \cj{\ft \w_{n_3}}), 
 \end{align*}

\noi
where $\w$ is the solution to the linearized equation
for \eqref{XNLS1} around the solution $w$ to \eqref{XNLS1} with $w|_{t = 0} = u_0$.
Hence, we have
\begin{align}
\big\|DF(t)|_{u_0}(\w(0)) \big\|_{\l^\infty_{\bar n}}
\les t \| w(t) \|_{L^2} \|\w(t)\|_{L^2}.
\label{Rem1}
 \end{align}

\noi
By combining this with (the proof of) Lemma \ref{LEM:Xnonlin}, 
we obtain 
\begin{align}
\big\| \jb{\dx}^{\frac{1}{2}+}DK_1(t)|_{u_0}(\w(0)) \big\|_{H^{s}} 
 \les
\sum_{j = 0}^4 t^j \sup_{t' \in [0, t]} \|w(t')\|_{H^{s-\frac 12-}}^{2j+2}\|\w(t')\|_{H^{s-\frac 12-}}
\label{Rem2}
\end{align}

\noi
for 
$\frac 34 < s \leq 1$.  Note that when differentiation hits
the complex exponentials, \eqref{Rem1} increases the value of $j$
by 1 in the statement of Lemma \ref{LEM:Xnonlin}
and hence we needed to include $j = 4$ in~\eqref{Rem2}.
Once we have \eqref{Rem2}, one can follow the argument in \cite{OTz1}
and prove  $DK_1(t)|_{u_0} \in HS(H^s(\T))$
for any $u_0 \in B_R\subset H^{s-\frac{1}{2}-}(\T)$.

\end{remark}

\begin{ackno}\rm
T.O.~was supported by the European Research Council (grant no.~637995 ``ProbDynDispEq'').
Y.T.~ was partially supported by JSPS KAKENHI Grant-in-Aid for Scientific Research (B) (17H02853) and Grant-in-Aid for Exploratory Research (16K13770).
\end{ackno}

\end{document}